\documentclass{birkjour}
\usepackage{graphics}
\usepackage{epsfig}
\usepackage{amsfonts}
\usepackage{amssymb}
\usepackage{amsthm}
\usepackage{newlfont}
 \newtheorem{thm}{Theorem}[section]

 \newtheorem{prop}[thm]{Proposition}
 \theoremstyle{definition}
 
 \theoremstyle{remark}

 \newtheorem*{ex}{Example}
 \numberwithin{equation}{section}

\begin{document}

%
%
%
%
%
%
%
%
%

\title[Generalized M\"{o}bius Ladder and Its Metric Dimension]
 {\begin{center}Generalized M\"{o}bius Ladder and Its Metric Dimension \end{center}}

\author{Ma Hongbin}
\address{School of Automation, Beijing Institute of Technology, Beijing-China}
\email{mathmhb@qq.com}
\author{Muhammad Idrees}
\address{School of Automation, Beijing Institute of Technology, Beijing-China}
\email{idrees@bit.edu.cn}
\author{Abdul Rauf Nizami}
\address{Abdus Salam School of Mathematical Sciences, GC University, Lahore-Pakistan}
\email{arnizami@sms.edu.pk}

\author{Mobeen Munir}
\address{Division of Science and Technology, University of Education, Lahore-Pakistan}
\email{mmunir@ue.edu.pk}
\maketitle

\begin{abstract}
In this paper we introduce generalized M\"{o}bius ladder $M_{m,n}$ and give its metric dimension. Moreover, it is observed that, depending on even and odd values of $m$ and $n$, it has two subfamilies with constant metric dimensions.
\end{abstract}

\pagestyle{myheadings}
\markboth{\centerline {\scriptsize
Hongbin, Idrees, Nizami, and Mobeen}} {\centerline {\scriptsize
 Generalized M\"{o}bius Ladder and Its Metric Dimension }}
\section{Introduction}\label{sec1}
The concepts of metric dimension and resolving set were introduced by Slater in \cite{slater:75, slater:98} and studied independently by Harary and Melter in \cite{harary-melter:76}. Since then the resolving sets have been widely investigated, as you can see in \cite{caceres:07, caceres:05, javaid:08, poisson:02, tomescu:07, tomescu-imran:09}.\\

\noindent Applications of metric dimension to the navigation of robots in networks are discussed in \cite{khuller-ragh:94}, to chemistry in \cite{chartrand:00}, and to image processing in \cite{melter-tomescu:84}.\\

\noindent A \emph{graph} $G$ is a pair $(V(G),E(G))$, where $V$ is the set of vertices and $E$ is the set of edges. A \emph{path} from a vertex $v$ to a vertex $w$ is a sequence of vertices and edges that starts from $v$ and stops at $w$. The number of edges in a path is the \emph{length} of that path. A graph is said to be \emph{connected} if there is a path between any two of its vertices. The \emph{distance} $d(u,v)$ between two vertices $u,v$ of a connected graph $G$ is the length of a shortest path between them.
\begin{center}
\begin{minipage}{5cm}
\centering  \epsfig{figure=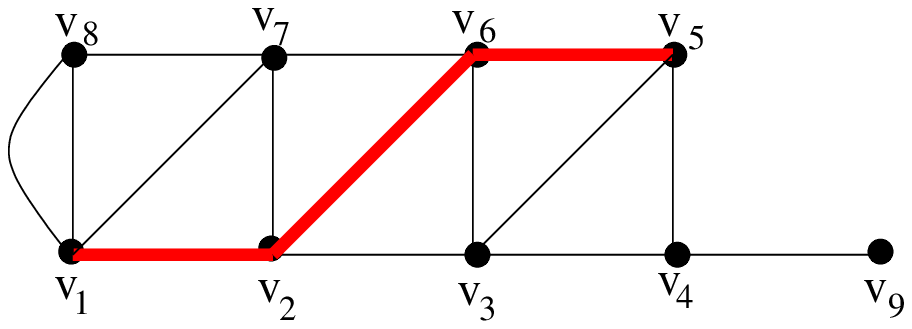,
height=2.0cm}
\par \tiny{A connected graph with a highlighted shortest path from $v_{1}$ to $v_{5}$}
\end{minipage}
\end{center}

\noindent Let $W=\{ w_{1},w_{2},\ldots,w_{k}\}$ be an ordered set of vertices of $G$ and let $v$ be a vertex of $G$. The representation $r(v|_{W})$ of $v$ with respect to $W$ is the $k$-tuple $(d(v,w_{1}),d(v,w_{2}),\ldots,d(v,w_{k}))$. If distinct vertices of $G$ have distinct representations with respect to $W$, then $W$ is called a \emph{resolving set} for $G$ [1]. A resolving set of minimum cardinality is called  a \emph{basis} of $G$; the number of elements in this basis is the \emph{metric dimension} of $G$, $\dim(G)$. A family $\mathcal{G}$ of connected graphs is said to have constant metric dimension if it is independent of any choice of member of that family.\\

\noindent The metric dimension of wheel $W_{n}$ is determined by Buczkowski et al. \cite{buczkowski:03}, of fan $f_{n}$ by Caceres et al. \cite{caceres:05, caceres:07}, and of Jahangir graph $J_{2n}$ by Tomescu et al. \cite{tomescu:07}. Chartrand et al. \cite{chartrand:00} proved that the family of path $P_{n}$ has the constant metric dimension 1. Javaid et al. proved in \cite{javaid:08} that the plane graph antiprism $A_{n}, n\geq5$ constitutes a family of regular graphs with constant metric dimension 3. The metric dimensions of some classes of plane graphs and convex polytopes have been studies in \cite{imran:3}, of generalized Petersen graphs $P(n,3)$ in \cite{imran:1}, of some rotationally-symmetric graphs in \cite{imran:2}. The part of the metric dimension of the M\"{o}bius ladder $M_{n}$ is determined in \cite{murtaza:12}, while the remaining part is determined in \cite{Mobeen-nizami:17}. \\

\noindent This paper is organized as follows: The generalized M\"{o}bius Ladder is introduced in Section~\ref{sec2}, its metric dimension is given in Section~\ref{sec3}, and the examples are given in Section~\ref{sec4}. The conclusive remarks are given in the last section.
\section{Generalized M\"{o}bius Ladder}\label{sec2}
Consider the Cartesian product $P_{m}\times P_{n}$ of paths $P_{m}$ and $P_{n}$ with vertices $u_{1},u_{2},\ldots,u_{m}$ and $v_{1},v_{2},\ldots,u_{n}$, respectively. Take a $180^{o}$ twist and identify the vertices $(u_{1},v_{1}),(u_{1},v_{2}),\ldots,(u_{1},v_{n})$ with the vertices $(u_{m},v_{n})$, $(u_{m},v_{n-1})$, $\ldots,(u_{m},v_{1})$, respectively, and identify the edge $\big((u_{1},i)$, $(u_{1},i+1)\big)$ with the edge $\big((u_{m},v_{n+1-i})$, $(u_{m},v_{n-i})\big)$, where $1\leq i\leq n-1$. What we receive is the generalized M\"{o}bius ladder $M_{m,n}$. You may observe that we receive the usual M\"{o}bius ladder for $n=2$ and for any odd integer $m\geq 4$. You can see $M_{7,3}$ in the following figure.
 \begin{center}
\begin{minipage}{5cm}
\centering  \epsfig{figure=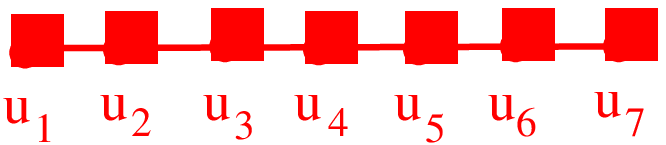, height=.6cm}
\par $P_{7}$
\end{minipage}
\begin{minipage}{5cm}
\centering  \epsfig{figure=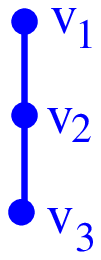, height=2.0cm}
\par $P_{3}$
\end{minipage}
\end{center}
\begin{center}
\begin{minipage}{10cm}
\centering  \epsfig{figure=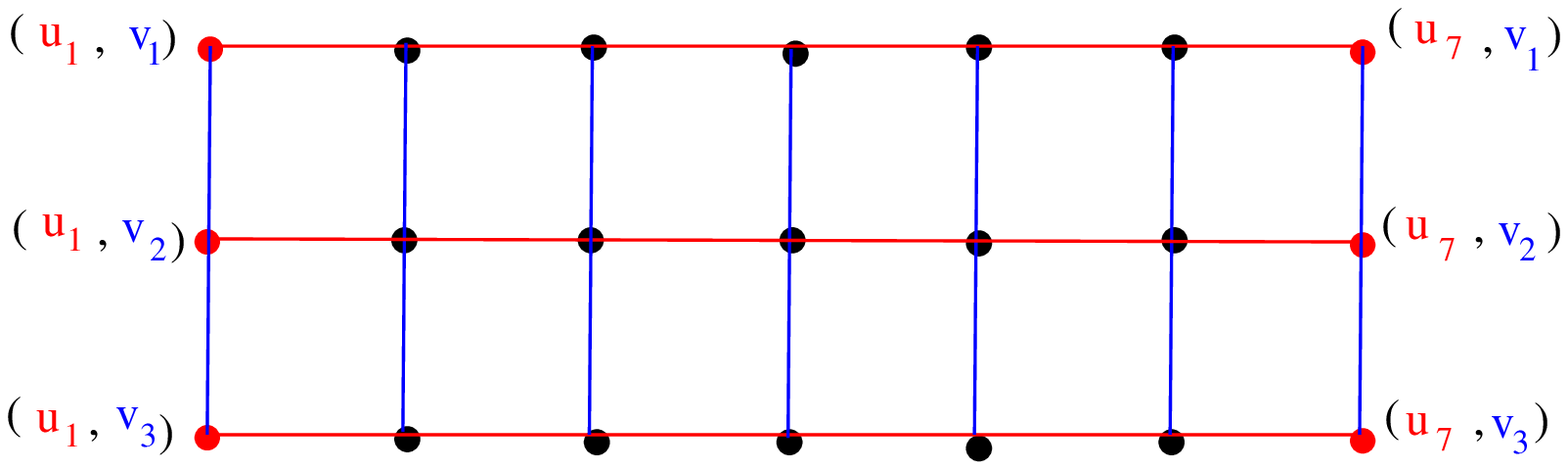, height=3.0cm}
\par $P_{7}\times P_{3}$
\end{minipage}
\end{center}
For brevity we shall use the symbol $v_{ij}$ (or simply $ij$) to represent the vertex $(u_{i},v_{j})$ of $M_{m,n}$, as you can see in the figure:
\begin{center}
\begin{minipage}{10cm}
\centering  \epsfig{figure=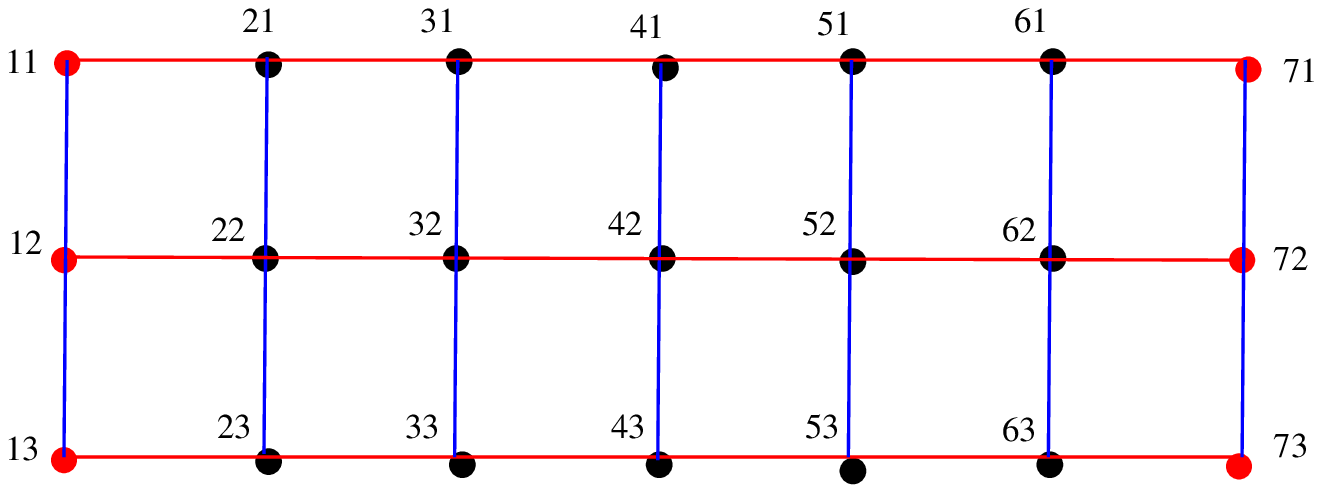, height=3.0cm}
\par $P_{7}\times P_{3}$ with complete simple labels
\end{minipage}
\end{center}
The generalized M\"{o}bius ladder obtained from $P_{7}\times P_{3}$ is:
\begin{center}
\begin{minipage}{10cm}
\centering  \epsfig{figure=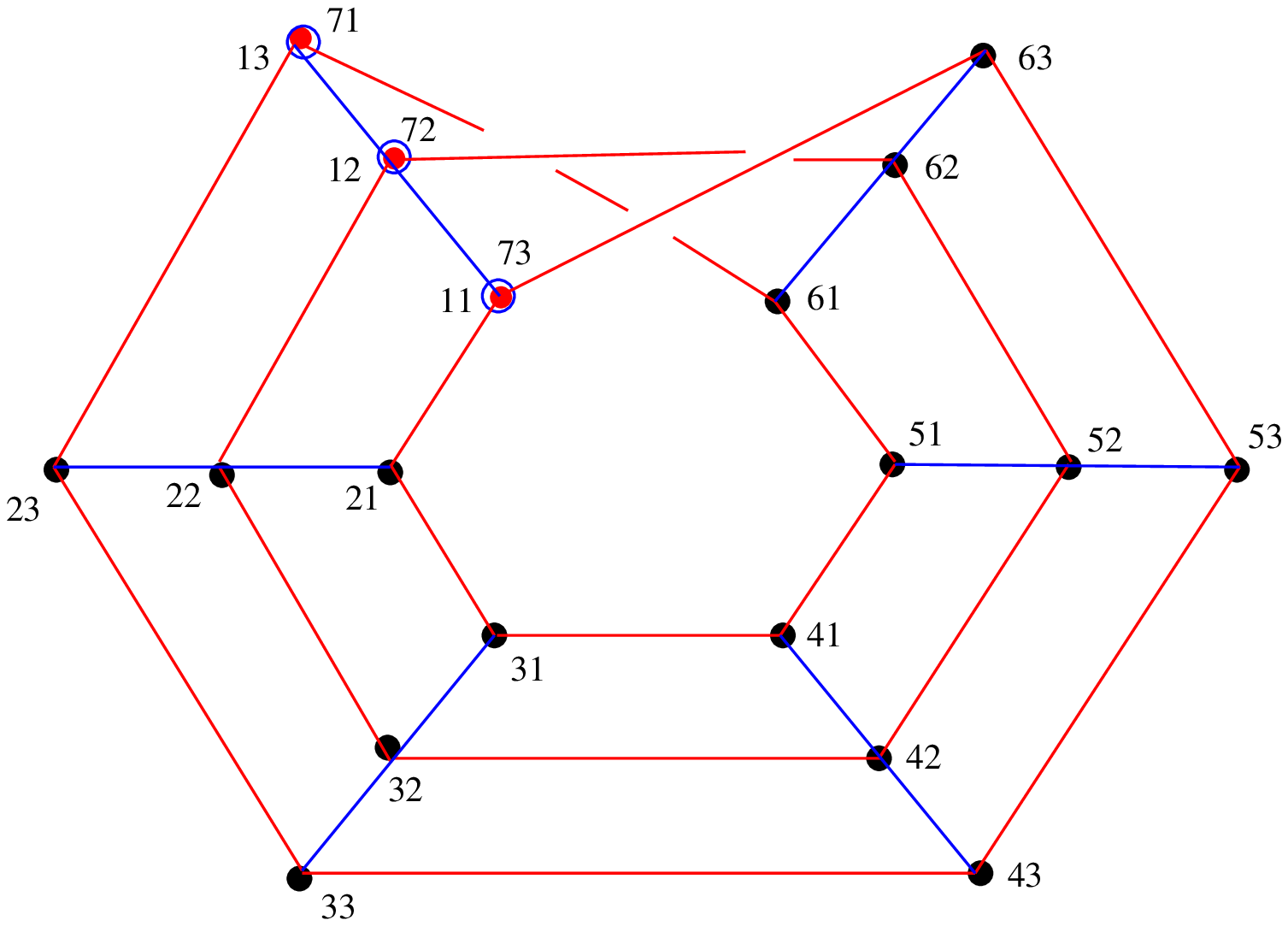, height=4.0cm}
\par $M_{7,3}$
\end{minipage}
\end{center}

\section{The Main Results}\label{sec3}
This section contains the metric dimension of the generalized M\"{o}bius ladder $M_{m,n}$. The results confirm that $M_{m,n}$ has two subfamilies with constant metric dimensions.
\begin{thm}\label{thm3.1}
The metric dimension of $M_{m,n}, m-n\geq 3,$ is 3 when one of $m$ and $n$ is even and other is odd.
\end{thm}
\begin{proof}
We claim that the resolving set in this case is $W=\{v_{1,1},v_{1,n},v_{m-1,1}\}$. It means the distance vectors corresponding to the vertices of $W$ are distinct, and no set with less than 3 vertices serves as a resolving set. In order to find distance vectors we involve two parameters, $q$ and $i$, and depending on their different values we divide the entries of distance vectors into three steps:\\

\noindent \textbf{Step I.} [Distances of $v_{1,1}$ with all vertices of $M_{m,n}$] In this case for each value of $q\in\{1,\ldots, n\}$ the parameter $i$ varies from $1$ to $m-1$. The entries of distance vectors are
 $$ d(v_{1,1},v_{i,q})=\left\{
                  \begin{array}{ll}
                    i+q-2 & 1\leq i\leq \frac{1}{2}(m+n-2q+1)\\
                    m+n-q-i & \frac{1}{2}(m+n-2q+3)\leq i\leq m-1\\
                  \end{array}
                \right.$$

\noindent \textbf{Step II.} [Distances of $v_{1,n}$ with all vertices of $M_{m,n}$] Here, again, for each value of $q\in\{1,\ldots, n\}$ the parameter $i$ varies from $1$ to $m-1$, and we get $ d(v_{1,n},v_{i,q})=d(v_{1,1},v_{i,n+1-q})$. \\

\noindent \textbf{Step III.} [Distances of $v_{m-1,1}$ with all vertices of $M_{m,n}$] In the following we have two parts:\\
\textbf{\emph{a}}) For $q=1$ we have
$$ d(v_{m-1,1},v_{i,q})=\left\{
                  \begin{array}{ll}
                    i+n-1 & 1\leq i\leq \frac{1}{2}(m-n-1)\\
                    m-1-i & \frac{1}{2}(m-n+1), 1\leq i\leq m-1\\
                  \end{array}
                \right.$$
\textbf{\emph{b}}) For for each value of $q\in\{2,\ldots, n\}$ the parameter $i$ varies from $1$ to $m-1$, and we get
$d(v_{m-1,1},v_{i,q})=d(v_{1,1},v_{i,n+2-q})$.\\

\noindent Now we show that no set with less than three vertices is a resolving set. For this it is enough to show that the removal of a single vertex from any resolving set of three vertices does resolve the graph anymore. For easy understanding let us take $W=\{v_{1,1},v_{1,n},v_{m-1,1}\}$ as the resolving set; in this case three different possibilities arise:\\
\textbf{Possibility I.} If we take $W_{1}=\{v_{1,1},v_{1,n}\}$, then $d(v_{i,j}|_{W_{1}})=d(v_{m-i+1,n-j+1}|_{W_{1}})$; here for each value of $i\in\{2,\ldots,m-1\}$ the value of $j$ varies from $1$ to $n$.\\
\textbf{Possibility II.} If we take $W_{2}=\{v_{1,1},v_{m-1,1}\}$, then $d(v_{i,j}|_{W_{2}})=d(v_{m-i,n-j+2}|_{W_{2}})$; here for each value of $i\in\{1,\ldots,m-1\}$ the value of $j$ varies from $2$ to $n$.\\
\textbf{Possibility III.} If we take $W_{3}=\{v_{1,n},v_{m-1,1}\}$, then $d(v_{i,j}|_{W_{3}})=d(v_{i+1,j+1}|_{W_{3}})$; here $i=j$ with $1\leq i,j\leq n-1$.\\
\end{proof}

\begin{thm}\label{thm3.2}
The metric dimension of $M_{m,n}, m-n\geq 4,$ is 4 when $m$ and $n$ are both even or odd.
\end{thm}
\begin{proof}
 Here the resolving set is $W=\{v_{1,1},v_{1,n},v_{m-1,1},v_{m-1,n}\}$. It means the distance vectors corresponding to these vertices are distinct, and no set with less than 4 vertices serves as a resolving set. Here we involve two parameters, $q$ and $i$, and depending on their different values we divide the entries of distance vectors into four steps:\\

\noindent \textbf{Step I.} [Distances of $v_{1,1}$ with all vertices of $M_{m,n}$]  In this case for each value of $q\in\{1,\ldots, n\}$ the parameter $i$ varies from $1$ to $m-1$. The entries of distance vectors are
 $$ d(v_{1,1},v_{i,q})=\left\{
                  \begin{array}{ll}
                    i+q-2 & 1\leq i\leq \frac{1}{2}(m+n-2q+2)\\
                    m+n-q-i & \frac{1}{2}(m+n-2q+4) 1\leq i\leq m-1\\
                  \end{array}
                \right.$$

\noindent \textbf{Step II.} [Distances of $v_{1,n}$ with all vertices of $M_{m,n}$] Here, again, for each value of $q\in\{1,\ldots, n\}$ the parameter $i$ varies from $1$ to $m-1$, and we get $ d(v_{1,n},v_{i,q})=d(v_{1,1},v_{i,n+1-q})$. \\

\noindent \textbf{Step III.}  [Distances of $v_{m-1,1}$ with all vertices of $M_{m,n}$] In the following we have two parts:\\
\textbf{\emph{a}}) For $q=1$ we have
$$ d(v_{m-1,1},v_{i,q})=\left\{
                  \begin{array}{ll}
                    i+n-1 & 1\leq i\leq \frac{1}{2}(m-n)\\
                    m-1-i & \frac{1}{2}(m-n+2), 1\leq i\leq m-1\\
                  \end{array}
                \right.$$
\textbf{\emph{b}}) For for each value of $q\in\{2,\ldots, n\}$ the parameter $i$ varies from $1$ to $m-1$, and we get
$d(v_{m-1,1},v_{i,q})=d(v_{1,1},v_{i,n+2-q})$.\\

\noindent \textbf{Step IV.} Here for each value of $q\in\{1,\ldots, n\}$ the parameter $i$ varies from $1$ to $m-1$, and we get $ d(v_{m-1,n},v_{i,q})=d(v_{m-1,1},v_{i,n+1-q})$.\\

\noindent Now we show that no set with less than four vertices is a resolving set. For this it is enough to show that the removal of a single vertex from any resolving set of four vertices does resolve the graph anymore. For easy understanding let us take $W=\{v_{1,1},v_{1,n},v_{m-1,1},v_{m-1,n}\}$ as the resolving set. In this case four different possibilities arise:\\
\textbf{Possibility I.} If we take $W_{1}=\{v_{1,1},v_{1,n},v_{m-1,1}\}$, then $d(v_{\frac{m-n+2}{2},1}|_{W_{1}})=d(v_{\frac{m+n}{2},n}|_{W_{1}})$.\\
\textbf{Possibility II.} If we take $W_{2}=\{v_{1,n},v_{m-1,1},v_{m-1,n}\}$, then $d(v_{\frac{m-n}{2},1}|_{W_{2}})=d(v_{\frac{m+n-2}{2},n}|_{W_{2}})$.\\
\textbf{Possibility III.} If we take $W_{3}=\{v_{1,1},v_{m-1,1},v_{m-1,n}\}$, then $d(v_{\frac{m-n}{2},2}|_{W_{3}})=d(v_{\frac{m-n+2}{2},1}|_{W_{3}})$.\\
\textbf{Possibility IV.} If we take $W_{4}=\{v_{1,1},v_{1,n},v_{m-1,n}\}$, then $d(v_{\frac{m-n+2}{2},2}|_{W_{4}})=d(v_{\frac{m-n+4}{2},1}|_{W_{4}})$.\\
\end{proof}
\section{Examples}\label{sec4}
In this section two examples are presented, one related to Theorem~\ref{thm3.1} and second related to Theorem~\ref{thm3.2}.
\begin{ex}
Consider $M_{7,4}$. The resolving set in this case is $W=\{v_{1,1},v_{1,4},v_{6,1}\}$.
\begin{center}
\begin{minipage}{10cm}
\centering  \epsfig{figure=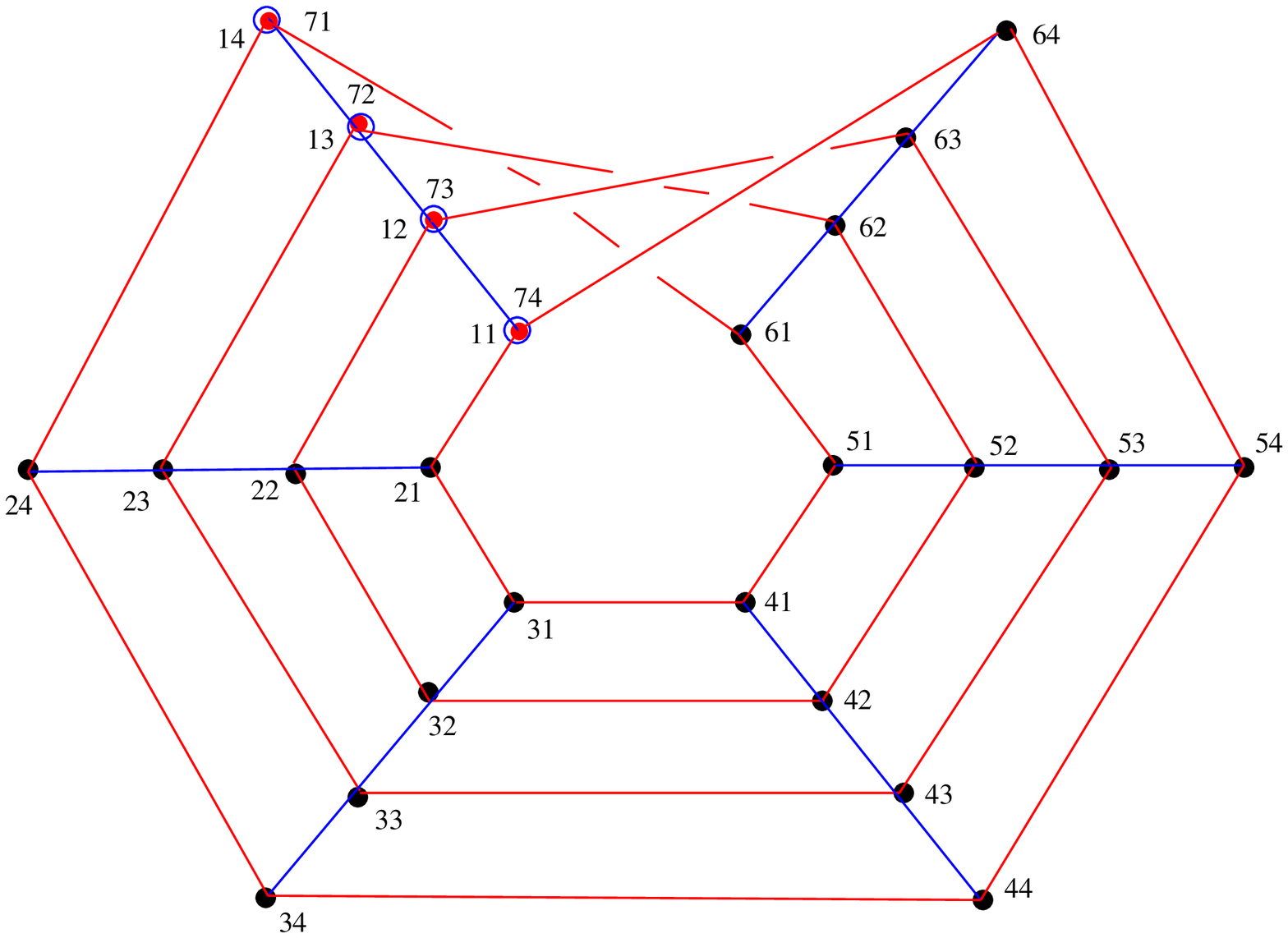, height=4.5cm}
\par $M_{7,4}$
\end{minipage}
\end{center} The full table of size $24\times 24$ representing the distances among the vertices of $M_{7,4}$ is  split up into two sub-tables, each having size $24\times 12$, and the distances corresponding to resolving vertices are given in bold form:
 \begin{center}
 \begin{tabular}{c|cccccccccccc}
 &$v_{11}$&$v_{12}$&$v_{13}$&$v_{14}$&$v_{21}$&$v_{22}$&$v_{23}$&$v_{24}$&$v_{31}$ & $v_{32}$ & $v_{33}$&$v_{34}$\\ \cline{1-13}
   $v_{11}$&\textbf{0}&\textbf{1}&\textbf{2}&\textbf{3}&\textbf{1}&\textbf{2}&\textbf{3}&\textbf{4}&\textbf{2}&\textbf{3}&\textbf{4}&\textbf{4} \\
   $v_{12}$&1&0&1&2&2&1&2&3&3&2&3&4 \\
   $v_{13}$&2&1&0&1&3&2&1&2&4&3&2&3 \\
   $v_{14}$&\textbf{3}&\textbf{2}&\textbf{1}&\textbf{0}&\textbf{4}&\textbf{3}&\textbf{2}&\textbf{1}&\textbf{5}&\textbf{4}&\textbf{3}&\textbf{2} \\
   $v_{21}$&1&2&3&4&0&1&2&3&1&2&3&4 \\
   $v_{22}$&2&1&2&3&1&0&1&2&2&1&2&3 \\
   $v_{23}$&3&2&1&2&2&1&0&1&3&2&1&2 \\
   $v_{24}$&4&3&2&1&3&2&1&0&4&3&2&1 \\
   $v_{31}$&2&3&4&4&1&2&3&4&0&1&2&3 \\
   $v_{32}$&3&2&3&4&1&2&3&4&1&0&1&2 \\
   $v_{33}$&4&3&2&3&3&2&1&2&2&1&0&1 \\
   $v_{34}$&4&4&3&2&4&3&2&1&3&2&1&0 \\
   $v_{41}$&3&4&4&3&2&3&4&4&1&2&3&4 \\
   $v_{42}$&4&3&4&4&3&2&3&4&2&1&2&3 \\
   $v_{43}$&4&3&3&4&4&3&2&3&3&2&1&2 \\
   $v_{44}$&3&4&4&3&4&4&3&2&4&3&2&1 \\
   $v_{51}$&4&4&3&2&3&4&4&3&2&3&4&4 \\
   $v_{52}$&4&3&2&3&4&3&3&4&3&2&3&4 \\
   $v_{53}$&3&2&3&4&4&3&3&4&4&3&2&3 \\
   $v_{54}$&2&3&4&4&3&4&4&3&4&4&3&2 \\
   $v_{61}$&\textbf{4}&\textbf{3}&\textbf{2}&\textbf{1}&\textbf{4}&\textbf{5}&\textbf{4}&\textbf{3}&\textbf{3}&\textbf{4}&\textbf{5}&\textbf{4} \\
   $v_{62}$&3&2&1&2&4&3&2&3&4&3&3&4 \\
   $v_{63}$&2&1&2&3&3&2&3&4&4&3&3&4 \\
   $v_{64}$&1&2&3&4&2&3&4&4&3&4&4&3 \\
      \hline
 \end{tabular}
 \end{center}
 \begin{center}
 \begin{tabular}{c|cccccccccccc}
 &$v_{41}$&$v_{42}$&$v_{43}$&$v_{44}$&$v_{51}$&$v_{52}$&$v_{53}$&$v_{54}$&$v_{61}$ & $v_{62}$ & $v_{63}$&$v_{64}$\\ \cline{1-13}
   $v_{11}$&\textbf{3}&\textbf{4}&\textbf{4}&\textbf{3}&\textbf{4}&\textbf{4}&\textbf{3}&\textbf{2}&\textbf{4}&\textbf{3}&\textbf{2}&\textbf{1} \\
   $v_{12}$&4&3&3&4&4&3&2&3&3&2&1&2 \\
   $v_{13}$&4&3&3&4&3&2&4&2&1&2&2&3 \\
   $v_{14}$&\textbf{3}&\textbf{4}&\textbf{4}&\textbf{3}&\textbf{2}&\textbf{3}&\textbf{4}&\textbf{4}&\textbf{1}&\textbf{2}&\textbf{3}&\textbf{4} \\
   $v_{21}$&2&3&4&4&3&4&4&3&4&4&3&2 \\
   $v_{22}$&3&2&3&4&4&3&3&4&4&3&2&3 \\
   $v_{23}$&4&3&2&3&4&3&3&4&3&2&3&3 \\
   $v_{24}$&4&4&3&2&3&4&4&3&2&3&4&4 \\
   $v_{31}$&1&2&3&4&2&3&4&4&3&4&4&3 \\
   $v_{32}$&2&1&2&3&3&2&3&4&4&3&4&4 \\
   $v_{33}$&3&2&1&2&4&3&2&3&4&4&3&4 \\
   $v_{34}$&4&3&2&1&4&4&3&2&3&4&4&3 \\
   $v_{41}$&0&1&2&3&1&2&3&4&2&3&4&4 \\
   $v_{42}$&1&0&1&2&2&1&2&3&3&2&3&4 \\
   $v_{43}$&2&1&0&1&3&2&1&2&4&3&2&3 \\
   $v_{44}$&3&2&1&0&4&3&2&1&4&4&3&2 \\
   $v_{51}$&1&2&3&4&0&1&2&3&1&2&3&4 \\
   $v_{52}$&2&1&2&3&1&0&1&2&2&1&2&3 \\
   $v_{53}$&3&2&1&2&2&1&0&1&3&2&1&2 \\
   $v_{54}$&4&3&2&1&3&2&1&0&4&3&2&1 \\
   $v_{61}$&\textbf{2}&\textbf{3}&\textbf{4}&\textbf{5}&\textbf{1}&\textbf{2}&\textbf{3}&\textbf{4}&\textbf{0}&\textbf{1}&\textbf{2}&\textbf{3} \\
   $v_{62}$&3&2&3&4&2&1&2&3&1&0&1&2 \\
   $v_{63}$&4&3&2&3&3&2&1&2&2&1&0&1 \\
   $v_{64}$&4&4&3&2&4&3&2&1&3&2&1&0 \\
      \hline
 \end{tabular}
 \end{center}
\end{ex}
\begin{ex}
The resolving set of $M_{10,2}$ is $W=\{v_{1,1},v_{1,2},v_{9,1},v_{9,2}\}$. Here, again, the full table of size $18\times 18$ representing the distances among the vertices of $M_{7,4}$ is  split up into two sub-tables, each having size $18\times 9$:
 \begin{center}
 \begin{tabular}{c|cccccccccccc}
 &$v_{11}$&$v_{12}$&$v_{21}$&$v_{22}$&$v_{31}$&$v_{32}$&$v_{41}$&$v_{42}$&$v_{51}$\\ \cline{1-10}
   $v_{11}$&\textbf{0}&\textbf{1}&\textbf{1}&\textbf{2}&\textbf{2}&\textbf{3}&\textbf{3}&\textbf{4}&\textbf{4} \\
   $v_{12}$&\textbf{1}&\textbf{0}&\textbf{2}&\textbf{1}&\textbf{3}&\textbf{2}&\textbf{4}&\textbf{3}&\textbf{5} \\
   $v_{21}$&1&2&0&1&1&2&2&3&3 \\
   $v_{22}$&2&1&1&0&2&1&3&2&4 \\
   $v_{31}$&2&3&1&2&0&1&1&2&2 \\
   $v_{32}$&3&2&2&1&1&0&2&1&3 \\
   $v_{41}$&3&4&2&3&1&2&0&1&1 \\
   $v_{42}$&4&3&3&2&2&1&1&0&2 \\
   $v_{51}$&4&5&3&4&2&3&1&2&0 \\
   $v_{52}$&5&4&4&3&3&2&2&1&1 \\
   $v_{61}$&5&4&4&5&3&4&2&3&1 \\
   $v_{62}$&4&5&5&4&4&3&3&2&2 \\
   $v_{71}$&4&3&5&4&4&5&3&4&2 \\
   $v_{72}$&3&4&4&5&5&4&4&3&3 \\
   $v_{81}$&3&2&4&3&5&4&4&5&3 \\
   $v_{82}$&2&3&3&4&4&5&5&4&4 \\
   $v_{91}$&\textbf{2}&\textbf{1}&\textbf{3}&\textbf{2}&\textbf{4}&\textbf{3}&\textbf{5}&\textbf{4}&\textbf{4} \\
   $v_{92}$&\textbf{1}&\textbf{2}&\textbf{2}&\textbf{3}&\textbf{3}&\textbf{4}&\textbf{4}&\textbf{5}&\textbf{5} \\
   \hline
 \end{tabular}
 \end{center}
 \begin{center}
 \begin{tabular}{c|cccccccccccc}
 &$v_{52}$&$v_{61}$&$v_{62}$&$v_{71}$&$v_{72}$&$v_{81}$&$v_{82}$&$v_{91}$&$v_{92}$\\ \cline{1-10}
   $v_{11}$&\textbf{5}&\textbf{5}&\textbf{4}&\textbf{4}&\textbf{3}&\textbf{3}&\textbf{2}&\textbf{2}&\textbf{1} \\
   $v_{12}$&\textbf{4}&\textbf{4}&\textbf{5}&\textbf{3}&\textbf{4}&\textbf{2}&\textbf{3}&\textbf{1}&\textbf{2} \\
   $v_{21}$&4&4&5&5&4&4&3&3&2 \\
   $v_{22}$&3&5&4&4&5&3&4&2&3 \\
   $v_{31}$&3&3&4&4&5&5&4&4&3 \\
   $v_{32}$&2&4&3&5&4&4&5&3&4 \\
   $v_{41}$&2&2&3&3&4&4&5&5&4 \\
   $v_{42}$&1&3&2&4&3&5&4&4&5 \\
   $v_{51}$&1&1&2&2&3&3&4&4&5 \\
   $v_{52}$&0&2&1&3&2&4&3&5&4 \\
   $v_{61}$&2&0&1&1&2&2&3&3&4 \\
   $v_{62}$&1&1&0&2&1&3&2&4&3 \\
   $v_{71}$&3&1&2&0&1&1&2&2&3 \\
   $v_{72}$&2&2&1&1&0&2&1&3&2 \\
   $v_{81}$&4&2&3&1&2&0&1&1&2 \\
   $v_{82}$&3&3&2&2&1&1&0&2&1 \\
   $v_{91}$&\textbf{5}&\textbf{3}&\textbf{4}&\textbf{2}&\textbf{3}&\textbf{1}&\textbf{2}&\textbf{0}&\textbf{1} \\
   $v_{92}$&\textbf{4}&\textbf{4}&\textbf{3}&\textbf{3}&\textbf{2}&\textbf{2}&\textbf{1}&\textbf{1}&\textbf{0} \\
   \hline
 \end{tabular}
 \end{center}
\end{ex}
\section{Conclusion}\label{sec5}
In this article we introduced generalized M\"{o}bius Ladder $M_{m,n}$ and proved that it has two subfamilies, each having a constant metric dimension. The metric dimension of $M_{m,n}$ is 3 when either $m$ is odd and $n$ is even or when $m$ is even and $n$ is odd; the metric dimension of $M_{m,n}$ is 4 when $m$ and $n$ are both odd  or when are both even. It is remarkable that the present results cover the results about the M\"{o}bius Ladder $M_{m}$ already presented in \cite{murtaza:12} and \cite{Mobeen-nizami:17} as subcases.

\end{document}